\documentclass[a4paper,12pt]{article}

\usepackage{amsmath, amssymb, amsthm}
\usepackage{mathtools}
\usepackage[dvips,final]{graphicx}

\usepackage{hyperref}
\hypersetup{
	colorlinks=true,
	linkcolor=blue,    
	urlcolor=cyan,
	citecolor=blue,
}

\usepackage{color}
\usepackage[dvipsnames]{xcolor}

\usepackage{titlesec}
\titleformat{\section}{\bfseries}{\thesection}{1em}{}
\titleformat{\subsection}{\itshape}{\thesubsection}{1em}{}

\numberwithin{equation}{section}

\newfont{\ctv}{msam10}

\newcommand{\bbox}{\mbox{\ctv \symbol{4}}}
\def\QED{{${}\hfill\bbox$}}
\newenvironment{pf}[1]{\par\vskip1mm{\noindent\it #1.}\ }{\QED\par
\vskip2mm}

\def\bpf{\begin{pf}}
\def\epf{\end{pf}}

\def\expe{\hbox{\rm e}}

\def\vrt{y}
\def\fpr{U_{k,R}}
\def\hukr{W_{k,R}}
\def\fhr{V_{k,R}}
\def\uipr{u_i^{(k,R)}}

\def\dd{\,\mathrm{d}}
\def\dive{\mathrm{\,div\,}}
\def\sign{\mathrm{\,sign}}

\def\supess{\mathop{\mathrm{\,sup\,ess}}}
\def\for{\mathrm{\ for\ }}
\def\ale{\mathrm{\ a.\,e.}}

\def\play{\mathfrak{p}}

\def\sumin{\sum_{i=1}^n}
\def\sumiz{\sum_{i=0}^n}
\def\sumim{\sum_{i=0}^{n-1}}

\def\om{^{(m)}}

\def\bfq{\boldsymbol{q}}
\def\bfn{\boldsymbol{n}}

\def\bfu{\boldsymbol{u}}

\def\bfnu{\boldsymbol{\nu}}
\def\bfsi{\boldsymbol{\sigma}}
\def\bfde{\boldsymbol{\delta}}
\def\bfxi{\boldsymbol{\xi}}

\def\real{\mathbb{R}}
\def\nat{\mathbb{N}}

\def\XX{\mathcal X}

\def\io{\int_{\Omega}}
\def\ipo{\int_{\partial\Omega}}

\def\be{\begin{equation}\label}
\def\ee{\end{equation}}

\def\ber{\begin{eqnarray}}
\def\eer{\end{eqnarray}}

\def\bers{\begin{eqnarray*}}
\def\eers{\end{eqnarray*}}

\def\bpf{\begin{pf}}
\def\epf{\end{pf}}

\newtheorem{theorem}{Theorem}[section]

\newtheorem{hypothesis}[theorem]{Hypothesis}
\newtheorem{proposition}[theorem]{Proposition}
\newtheorem{definition}[theorem]{Definition}

\begin{document}

\title{Deformable porous media with degenerate hysteresis\\in gravity field\thanks{The support from the European Union's Horizon Europe research and innovation programme under the Marie Sk\l odowska-Curie grant agreement No 101102708, from the M\v{S}MT grant 8X23001, and from the GA\v CR project 24-10586S is gratefully acknowledged.}
}

\author{Chiara Gavioli
\thanks{Faculty of Civil Engineering, Czech Technical University, Th\'akurova 7, CZ-16629 Praha 6, Czech Republic, E-mail: {\tt chiara.gavioli@cvut.cz}.}
\and Pavel Krej\v c\'{\i}
\thanks{Faculty of Civil Engineering, Czech Technical University, Th\'akurova 7, CZ-16629 Praha 6, Czech Republic, E-mail: {\tt Pavel.Krejci@cvut.cz}.}
\thanks{Institute of Mathematics, Czech Academy of Sciences, {\v{Z}}itn{\'a} 25, CZ-11567 Praha 1, Czech Republic, E-mail: {\tt krejci@math.cas.cz}.}
}

\date{}

\maketitle


\begin{abstract}
Hysteresis in the pressure-saturation relation in unsaturated porous media, which is due to surface tension on the liquid-gas interface, exhibits strong degeneracy in the resulting mass balance equation. Solutions to such degenerate equations have been recently constructed by the method of convexification even if the permeability coefficient depends on the hysteretic saturation. The model is extended here to the case that the solid matrix material is viscoelastic and that the system is coupled with a gravity driven moisture flux. The existence of a solution is proved by compact anisotropic embedding involving Orlicz spaces with respect to the time variable.

\bigskip


\medskip

\end{abstract}


\section{Introduction}

We pursue here the study started in \cite{colli,perme} of degenerate diffusion in unsaturated porous media filled with liquid and gas. Here we additionally take into account the deformations of the solid matrix produced by the penetrating humidity, and the effects of gravity on fluid diffusion as a generalization of the Richards equation (see \cite{bsch}). To be precise, our main modeling assumptions are the following:
\begin{enumerate}
	\item[{\rm (a)}]   {\rm The pressure-saturation relation exhibits hysteresis};
	\item[{\rm (b)}]    {\rm The solid skeleton is viscoelastic};
	\item[{\rm (c)}]    {\rm The permeability of the material depends on the moisture content};
	\item[{\rm (d)}]    {\rm A gravity-driven moisture flux takes part in the process}.
\end{enumerate}

Our model is inspired by the experimental evidence about hysteresis in hydrogeology in, e.g., \cite{alb,hol,pfb}. Figure~\ref{fi1} is taken from \cite{pfb} and shows rate-independent hysteresis between the logarithm soil suction $\psi$, which is a decreasing function of the pressure, and the volumetric water content $\theta$ which, up to a linear transformation, can be identified with the saturation. The main characteristic of the hysteresis relation in Figure~\ref{fi1} is that at turning points where the pressure changes the orientation from decreasing to increasing or vice versa, the starting slope is horizontal, which makes the diffusion problem degenerate. 

\begin{figure}[h]
	\begin{center}
		\includegraphics[width=8cm]{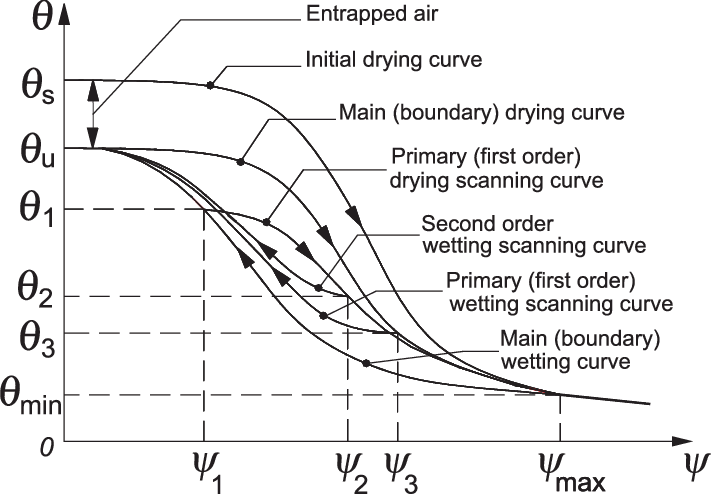}
		\vspace{5mm}
		\caption{Typical experimental hysteresis dependence in porous media between the logarithm soil suction $\psi$ and the volumetric water content $\theta$.}
		\label{fi1}
	\end{center}
\end{figure}

Mathematical investigation of diffusion problems with hysteresis goes back to A.~Visintin's pioneering monograph \cite{vis} presenting basic methods of solving PDEs with hysteresis. However, two problems have remained unsolved until recently: the degeneracy at turning points and the dependence of the diffusion coefficient on the saturation. We have proved in \cite{colli} that the degeneracy in the case of constant permeability can be overcome using a method developed there to convexify the hysteresis relation. In \cite{perme}, the convexification combined with a compact embedding theorem for anisotropic Orlicz spaces was shown to give a sufficient argument for the solvability also in the case of saturation-dependent permeability. Let us also mention some previous results on hysteresis dependence of the permeability coefficient in the non-degenerate case under additional regularization in time or space \cite{bavi1,bavi2,gosch,vis2}. To our knowledge, only E.~El Behi-Gornostaeva addressed in her thesis \cite{gorno} the full problem of hysteresis-dependent saturation without regularization in the non-degenerate case and proposed a method for the existence proof.

In this paper we extend the problem to the case of deformable porous media under gravity effects and full degeneracy of the saturation dependence. The process is driven by the mass conservation principle and the quasistatic mechanical equilibrium equation in the small deformations regime. We further simplify the system in order to make it mathematically tractable by assuming that shear stresses are negligible. The gravity term and the viscoelasticity of the deformable solid skeleton represent a new degree of complexity related to the fact that no obvious a priori upper bound for the solutions is available here. We show that such an estimate can be obtained using a hysteresis variant of the Moser iteration technique under an additional assumption on the admissible Preisach operators $G$, namely, that the Preisach density $\phi(x,r,v)$ in  \eqref{ge3} of $G$ decays sufficiently slowly at infinity, see Hypothesis~\ref{hrho}. The proof of the existence theorem is then carried out by a convexity and compactness argument developed in \cite{perme}.

The structure of the present paper is the following. In Section~\ref{mode} we set up the mathematical model for the phenomenon and show that it is consistent with the physical principles of mass and energy conservation. In Section~\ref{stat} we recall the definitions of the main concepts including convexifiable Preisach operators, and state the main existence Theorem~\ref{t1}. In Section~\ref{disc} we propose a time discretization scheme with time step $\tau>0$ and derive estimates independent of~$\tau$. A uniform upper bound for the time-discrete approximation is established in Section~\ref{unif} via Moser-type iterations. In Section~\ref{conv} we show that, similarly to \cite{perme}, we can derive the convexity estimate \eqref{dp6} for the time derivative of the pressure. As shown in the last Section~\ref{proof}, this is sufficient to let the time step $\tau$ tend to $0$ and to prove that the limit is the desired solution.


\section{Derivation of the model}\label{mode}

The physical quantities which appear in the model are listed below in Table~\ref{tab}.

\begin{table}[h]
	\begin{center}
		\caption{List of physical quantities.}
		\label{tab}
		\renewcommand\arraystretch{1.5}
		\begin{tabular}{ c ccc c  }
			\hline
			\textbf{Symbol} & \qquad & \textbf{Quantity} &\qquad &\textbf{Physical dimension}  \\ \hline
			$\rho$ && liquid mass density  && $\left[\frac{kg}{m^3}\right]$\\ 
			$\theta$ && volumetric water content && $\left[-\right]$\\ 
			$\bfq$ && liquid mass flux && $\left[\frac{kg}{m^2\, {s}}\right]$\\ 
			$\bfsi$ && stress tensor && $\left[\frac{kg}{m\, s^2}\right]$\\ 
			$\bfu$ && displacement vector && $\left[m\right]$\\ 
			$\mu$ && bulk elasticity modulus && $\left[\frac{kg}{m\, s^2}\right]$\\ 
			$\gamma$ && bulk viscosity && $\left[\frac{kg}{m\, s}\right]$\\ 
			$p$ && liquid pressure && $\left[\frac{kg}{m\, {s}^2}\right]$\\ 
			$p_0$ && standard pressure && $\left[\frac{kg}{m\, {s}^2}\right]$
			\\ 
			$g_0$ && gravity constant && $\left[\frac{m}{{s}^2}\right]$\\ 
			$\kappa$ && permeability && $\left[{s}\right]$
			\\ 
			$b^*$ && boundary permeability && $\left[\frac{s}{m}\right]$\\
			\hline
		\end{tabular}
	\end{center}
\end{table}

The fluid diffusion in a domain $\Omega \subset \real^3$ filled with a deformable unsaturated porous medium is driven by the liquid mass conservation principle
\begin{equation}\label{e1}
	\rho \theta_t = -\dive \bfq
\end{equation}
with constant mass density $\rho$, and by the mechanical equilibrium equation for the solid
\begin{equation}\label{s1}
	\dive \bfsi = 0 \ \ \mbox{ in }\ \Omega, \quad \bfsi\cdot \bfn(x) = 0 \ \ \mbox{ on }\ \partial\Omega,
\end{equation}
where $\bfn(x)$ denotes the unit outward normal vector to $\Omega$ at a point $x \in \partial\Omega$.
We write \eqref{s1} in variational form
\begin{equation}\label{s2}
	\io \bfsi : \nabla_s \bfxi \dd x = 0, \qquad \forall \bfxi \in W^{1,2}(\Omega; \real^3),
\end{equation}
where $\nabla_s$ denotes the symmetric gradient.

We assume that the deformations are small, so that $\dive \bfu$ is the relative local volume change, and that shear stresses can be neglected, that is,
\begin{equation}\label{s3}
	\bfsi = (\mu \dive \bfu + \gamma \dive \bfu_t - p)\bfde,
\end{equation}
where $\mu >0$ is a constant bulk elasticity modulus, $\gamma >0$ is a bulk viscosity modulus, and $\bfde = (\delta_{ij})$, $i,j = 1,2,3$ is the Kronecker symbol. Then \eqref{s2} can be written as
\begin{equation}\label{s3a}
	\io (\mu \dive \bfu + \gamma \dive \bfu_t - p) \dive \bfxi \dd x = 0, \qquad \forall \bfxi \in W^{1,2}(\Omega; \real^3). 
\end{equation}
Let $h \in L^2(\Omega)$ be arbitrary, and let $W$ be the solution to the problem
\begin{equation}\label{s4}
	-\Delta W = h, \qquad W = 0 \ \mbox{ on } \ \partial \Omega.
\end{equation}
Putting $\bfxi = \nabla W$ in \eqref{s3a}, we get the mechanical equilibrium equation in the form
\begin{equation}\label{s5}
	\mu \dive \bfu + \gamma \dive \bfu_t - p = 0
\end{equation}
a.e.  in $\Omega$.

Assume now that the relative liquid mass flux $\bfq - \rho \bfu_t$ is proportional to the liquid pressure gradient $\nabla p$, that is,
\begin{equation}\label{e2}
	\bfq - \rho \bfu_t = -\kappa\nabla p,
\end{equation}
with proportionality factor $\kappa = \kappa(x,\theta)>0$ depending on $x \in \Omega$ and $\theta$. The liquid pressure $p$ can be decomposed into two components
\begin{equation}\label{pres1}
	p = p_c + p_h,
\end{equation}
where $p_c$ is the capillary pressure and $p_h$ is the hydrostatic pressure. For the hydrostatic pressure we assume the classical relation
\begin{equation}\label{pres2}
	p_h = \rho g_0 \bfnu\cdot (x-x_0),
\end{equation}
where $\bfnu$ is the unit vector in the gravity direction and $x_0$ is a referential point.

We introduce the dimensionless normalized capillary pressure
\begin{equation}\label{pres3}
	u = \frac{p_c}{p_0}
\end{equation}
and follow the modeling hypotheses of \cite{alb,bsch} which consists in representing the $p_c \mapsto \theta$ hysteresis relation by the formula
\begin{equation}\label{e3}
	\theta = G[u],
\end{equation}
where $G$ is a Preisach operator defined below in Section~\ref{stat}. 

On the boundary $\partial\Omega$ of $\Omega$ we assume that the unit outward normal $\bfn(x)$ is defined almost everywhere and that the normal component of the relative flux $\bfq - \rho \bfu_t$ in \eqref{e2} is proportional to the difference of pressures $p$ inside and $p^*$ outside the body, that is,
\begin{equation}\label{bou}
	-\kappa(x,\theta)\nabla p\cdot \bfn(x) = b^*(x)(p-p^*) \ \ \mbox{ on }\ \partial\Omega.
\end{equation}
The weak formulation of the mass balance Eq \eqref{e1} then reads
\begin{equation}\label{mass}
	\io \Big(\rho (\theta_t + \dive \bfu_t) \vrt + \kappa(x,\theta)\nabla p\cdot\nabla \vrt\Big)\dd x + \ipo b^*(x)(p-p^*)\vrt\dd s(x) = 0 
\end{equation}
for every test function $\vrt \in W^{1,2}(\Omega) \cap L^\infty(\Omega)$.

Assume for the moment that there is no mass exchange with the exterior corresponding to the choice $b^*(x) = 0$ in \eqref{bou}, that is, 
\begin{equation}\label{e5}
	-\kappa(x,\theta)\nabla p\cdot \bfn(x) = 0
\end{equation}
on $\partial\Omega$. Putting $\vrt=1$ in \eqref{mass}, we get
\begin{equation}\label{e6}
	\frac{\dd}{\dd t} \io\rho (\theta + \dive \bfu)(x,t)\dd x =  0.
\end{equation}
The interpretation of \eqref{e6} is mass conservation. The term $\io \dive \bfu(x,t)\dd x$ describes the evolution of the volume represented by $\Omega$, and $\rho\io (\theta + \dive \bfu)(x,t)\dd x$ is the total water mass in $\Omega$ at time $t$. Naturally, if the volume increases and mass is conserved, then the saturation decreases and vice versa.

Consider now the energy balance. Under the boundary condition \eqref{e5}, no power is supplied to the system, and the total energy of the system should therefore decrease because all phenomena like viscosity, diffusion, and hysteresis dissipate energy. We put $\vrt = p/\rho$ in \eqref{mass} and obtain using \eqref{e5} that
\begin{equation}\label{ee7}
	\io \big(\theta_t + \dive \bfu_t\big) p\, \dd x + \io \frac{\kappa(x,\theta)}{\rho}|\nabla p|^2 \dd x = 0.
\end{equation}
For a Preisach operator $G$ in \eqref{e3} there exists a Preisach potential operator $V$ and a dissipation operator $D$ such that
\begin{equation}\label{e8g}
	G[u]_t u = V[u]_t +|D[u]_t| \ \ale
\end{equation}
Using \eqref{s5} and \eqref{pres1}--\eqref{pres3}, we can rewrite \eqref{ee7} in the form
\begin{align}\nonumber
	&\frac{\dd}{\dd t} \io\left(p_0 V[u] + \rho g_0 \theta\, \bfnu\cdot (x-x_0) + \frac{\mu}{2}|\dive\bfu|^2\right) \dd x\\[2mm]\label{e8}
	&\qquad + \io \left(p_0|D[u]_t| + \frac{\kappa(x,\theta)}{\rho}|\nabla p|^2 + \gamma |\dive\bfu_t|^2\right)\dd x = 0.
\end{align}
The energetic interpretation of \eqref{e8} is the following: The term
\begin{equation}\label{e9}
	\io \left(p_0|D[u]_t| + \frac{\kappa(x,\theta)}{\rho}|\nabla p|^2 + \gamma |\dive\bfu_t|^2\right)\dd x
\end{equation}
is the dissipation rate which is positive in agreement with the principles of thermodynamics, while
\begin{equation}\label{e10}
	\io\left(p_0 V[u] + \rho g_0 \theta\, \bfnu\cdot (x-x_0) + \frac{\mu}{2}|\dive\bfu|^2\right) \dd x
\end{equation}
is the potential energy of the system which is therefore decreasing as expected.

Putting $v = \dive \bfu - p_h/\mu$, from \eqref{s5} and \eqref{pres1}--\eqref{mass} we get the following system, consisting of a PDE coupled with an ODE, for the unknowns $u$ and $v$:
\begin{align}\label{pde0}
	&	\io \left((G[u] + v)_t \vrt + \frac{\kappa(x,G[u]) p_0}{\rho}\left(\nabla u + \frac{\rho g_0}{p_0}\bfnu\right)\cdot\nabla \vrt\right)\dd x + \ipo \frac{p_0 b^*(x)}{\rho}(u-u^*)\vrt\dd s(x)  = 0,\\ \label{pde1}
	& \qquad \qquad  \qquad \qquad  \qquad \qquad   \gamma v_t + \mu v - p_0 u = 0,
\end{align}
for every $\vrt \in W^{1,2}(\Omega) \cap L^\infty(\Omega)$, where we set $u^* = (p^*- p_h)/p_0$.

The existence proof for \eqref{pde0}-\eqref{pde1} is independent of the exact values of the physical constants. We therefore normalize all constants to $1$ for simplicity and consider the system
\begin{align}\label{pde}
	\io \Big((G[u] + v)_t \vrt + \kappa(x,G[u])&\left(\nabla u + \bfnu\right)\cdot\nabla \vrt\Big)\dd x + \ipo  b^*(x)(u-u^*)\vrt\dd s(x) = 0,\\ \label{pd0}
	&	v_t + v = u,
\end{align}
for every $\vrt \in W^{1,2}(\Omega)\cap L^\infty(\Omega)$, coupled with the initial conditions
\begin{equation}\label{ie3}
	u(x, 0) = u_0(x), \quad v(x,0) = v_0(x)\ \ale \mbox{ in } \ \Omega.
\end{equation}


\section{Statement of the problem}\label{stat}

We study here Problem \eqref{pde}--\eqref{ie3} in a bounded Lipschitzian domain $\Omega \subset \real^N$ and time interval $(0,T)$, where  $N \in \nat$ can be arbitrary. Recall first the definition of the Preisach operator introduced in \cite{prei} in the variational setting of \cite{book}.

\begin{definition}\label{dpr}
	Let $\lambda \in L^\infty(\Omega \times (0,\infty))$ be a given function with the following properties: 
	\begin{align}
		\label{ge6b}
		&\exists \Lambda>0 : \ \lambda(x,r) = 0\ \for r\ge \Lambda, \forall x \in \Omega,\\[2mm] \label{ge6}
		&\begin{aligned}
			\exists \bar{\lambda}>0 : \ |\lambda(x_1,r_1) - \lambda(x_2,r_2)| &\le \Big(\bar{\lambda}\,|x_1 - x_2| + |r_1 - r_2|\Big)\ \forall r_1, r_2 \in (0,\infty), \forall x_1,x_2 \in \Omega.
		\end{aligned}
	\end{align}
	For a given $r>0$, we call the {\em play operator with threshold $r$ and initial memory $\lambda$} the mapping which with a given function $u \in L^1(\Omega; W^{1,1}(0,T))$ associates the solution $\xi^r\in L^1(\Omega; W^{1,1}(0,T))$ of the variational inequality
	\begin{equation}\label{ge4a}
		|u(x,t) - \xi^r(x,t)| \le r, \quad \xi^r_t(x,t)(u(x,t) - \xi^r(x,t) - z) \ge 0 \ \ale \ \forall z \in [-r,r],
	\end{equation}
	with initial condition
	\begin{equation}\label{ge5}
		\xi^r(x,0) = \lambda(x,r) \ \ale,
	\end{equation}
	and we denote
	\begin{equation}\label{ge4}
		\xi^r(x,t) = \play_r[\lambda,u](x,t).
	\end{equation}
	Given a measurable function $\phi :\Omega\times(0,\infty)\times \real \to [0,\infty)$ and a constant $\bar G \in \real$, the Preisach operator $G$ is defined as a mapping $G: L^2(\Omega; W^{1,1}(0,T))\to L^2(\Omega; W^{1,1}(0,T))$ by the formula
	\begin{equation}\label{ge3}
		G[u](x,t) = \bar G + \ \int_0^\infty\int_0^{\xi^r(x,t)} \phi(x,r,v)\dd v\dd r.
	\end{equation}
	The Preisach operator is said to be {\em regular\/} if the density function $\phi$ of $G$ in \eqref{ge3} belongs to $L^\infty(\Omega\times (0,\infty )\times \real)$, and there exist constants $\phi_1,\bar{\phi} > 0$ and a decreasing function $\phi_0: \real\to \real$ such that for all $U>0$, all $x,x_1,x_2 \in \Omega$, and a.e.\ $(r,v)\in (0,U) \times (-U,U)$ we have
	\begin{align}
		&0 < \phi_0(U) < \phi(x,r,v) < \phi_1, \label{ge3a} \\[2mm]
		&|\phi(x_1,r,v) - \phi(x_2,r,v)| \le \bar{\phi}\,|x_1 - x_2|. \label{ge3b}
	\end{align}
\end{definition}

In applications, the natural physical condition $\theta=G[u] \in [0,1]$ is satisfied for each input function $u$ if and only if $\bar G \in (0,1)$ and
\begin{equation}\label{irho}
	\int_0^\infty \int_0^\infty \phi(x,r,v)\dd v\dd r \le 1-\bar G, \quad \int_0^\infty \int_0^\infty \phi(x,r,-v)\dd v\dd r \le \bar G,
\end{equation}
for a.e.\ $x\in \Omega$. However, for the existence result in Theorem~\ref{t1} only \eqref{ge3a} and \eqref{ge3b} are substantial.

Let us mention the following classical result (see \cite[Proposition~II.3.11]{book}).

\begin{proposition}\label{pc1}
	Let $G$ be a regular Preisach operator in the sense of Definition~\ref{dpr}. Then it can be extended to a Lipschitz continuous mapping $G: L^p(\Omega; C[0,T]) \to L^p(\Omega; C[0,T])$ for every $p \in [1,\infty)$.
\end{proposition}

The Preisach operator is rate-independent. Hence, for input functions $u(x,t)$ which are monotone in a time interval $t\in (a,b)$, a regular Preisach operator $G$ can be represented by a superposition operator $G[u](x,t) = B(x, u(x,t))$ with an increasing function $u \mapsto B(x, u)$ called a {\em Preisach branch\/}. Indeed, the branches may be different at different points $x$ and different intervals $(a,b)$. The branches corresponding to increasing inputs are said to be {\em ascending\/} (the so-called wetting curves in the context of porous media), the branches corresponding to decreasing inputs are said to be {\em descending\/} (drying curves).

\begin{definition}\label{dpc}
	Let $U>0$ be given. A Preisach operator is said to be {\em uniformly counterclockwise convex on $[-U,U]$\/} if for all inputs $u$ such that $|u(x,t)|\le U$ a.e., all ascending branches are uniformly convex and all descending branches are uniformly concave in the sense that there exists $\beta >0$ such that for every $u \in (-U,U)$ we have the implications
	$$
	u_t >0 \implies \frac{\partial^2 B}{\partial u^2} \ge \beta, \qquad u_t<0 \implies \frac{\partial^2 B}{\partial u^2} \le -\beta
	$$
	in the sense of distributions, see \cite{colli}.
	
	A regular Preisach operator $G$ is called {\em convexifiable\/} if for every $U>0$ there exist a uniformly counterclockwise convex Preisach operator $P$ on $[-U,U]$, positive constants $g_*(U),g^*(U),\bar{g}(U)$, and a twice continuously differentiable mapping $g:[-U,U] \to [-U,U]$ such that
	\begin{equation}\label{hg}
		g(0)=0, \quad 0<g_*(U) \le g'(u) \le g^*(U), \quad |g''(u)| \le \bar g(U)\ \ \forall u\in [-U,U],
	\end{equation}
	and $G = P\circ g$.
\end{definition}

A typical example of a uniformly counterclockwise convex operator is the so-called {\em Prandtl-Ishlinskii operator\/} characterized by positive density functions $\phi(x,r)$ independent of $v$, see \cite[Section~4.2]{book}. Operators of the form $P\circ g$ with a Prandtl-Ishlinskii operator $P$ and an increasing function $g$ are often used in control engineering because of their explicit inversion formulas, see \cite{al,viso,kk}. They are called the {\em generalized Prandtl-Ishlinskii operators\/} (GPI) and represent an important subclass of Preisach operators. Note also that for every Preisach operator $P$ and every Lipschitz continuous increasing function $g$, the superposition operator $G = P\circ g$ is also a Preisach operator, and there exists an explicit formula for its density, see \cite[Proposition~2.3]{error}. Another class of convexifiable Preisach operators is shown in \cite[Proposition~1.3]{colli}.

The technical hypotheses on the data in \eqref{pde}-\eqref{pd0} can be stated as follows.

\begin{hypothesis}\label{hy2}
	The boundary permeability $b^*$ belongs to $L^\infty(\partial\Omega)$, and is such that $b^*(x)\ge 0$ 
	a.e.\ and $\ipo b^*(x)\dd s(x) > 0$. The boundary source $u^*$ belongs to $L^\infty(\partial \Omega\times(0,T))$; in addition, $u^*_t \in L^2(\partial \Omega\times(0,T))$.
	The permeability $\kappa:\Omega \times \real \to \real$ is a bounded Lipschitz continuous function, more precisely, there exist constants $\bar\kappa>0$, $\kappa_1 > \kappa_0 >0$ such that for all $\theta, \theta_1, \theta_2\in \real$ and all $x, x_1, x_2 \in \Omega$ we have
	$$
	\kappa_0 \le \kappa(x,\theta) \le \kappa_1, \quad |\kappa(x_1,\theta_1) - \kappa(x_2,\theta_2)| \le \bar\kappa\big(|x_1 - x_2| + |\theta_1 - \theta_2|\big).
	$$
\end{hypothesis}

Note that we can rewrite \eqref{pde}-\eqref{pd0} for $\vrt \in W^{1,2}(\Omega)\cap L^\infty(\Omega)$ as
\begin{align*}
	\io \Big((G[u]_t + u-v)\vrt &+ \kappa(x,G[u])\left(\nabla u + \bfnu\right)\cdot\nabla \vrt\Big)\dd x + \ipo  b^*(x)(u-u^*)\vrt\dd s(x) = 0, \\
	&  v_t + v = u.
\end{align*}
Therefore, taking into account \eqref{ie3}, even a local solution to Problem \eqref{pde}--\eqref{ie3} may fail to exist if for example $\lambda(x,r) \equiv 0$ and $\!\dive\!\big(\kappa(x,G[u](x,0))\left(\nabla u_0(x) + \bfnu\right)\big)+v_0(x)-u_0(x) \ne 0$, and we need an initial memory compatibility condition which we state here following \cite{colli}. A more detailed discussion on this issue can be found in the introduction to \cite{colli}.

\begin{hypothesis}\label{hy1}
	Let the initial memory $\lambda$ and the Preisach density function $\phi$ be as in Definition~\ref{dpr}. The initial pressure $u_0$ belongs to $W^{2,\infty}(\Omega)$, the initial volume strain $v_0$ belongs to $L^\infty(\Omega)$, and there exist a constant $L>0$ and a function $r_0 \in L^\infty(\Omega)$ such that, for $\Lambda>0$ as in \eqref{ge6b}, $\supess_{x\in \Omega}|u_0(x)| \le \Lambda$ and for a.e.\ $x \in \Omega$ the following initial compatibility conditions hold:
	\begin{align}\label{c0}
		\lambda(x,0) = u_0(x) &\ \ale,\\ \label{c0a}
		\theta_0(x) = G[u](x,0) = \bar G + \ \int_0^\infty\int_0^{\lambda(x,r)} \phi(x,r,v)\dd v\dd r &\ \ale,\\ \label{c1}
		\frac1L \sqrt{\big|\!\dive\!\big(\kappa(x,\theta_0(x))\left(\nabla u_0(x) + \bfnu\right)\big)+v_0(x)-u_0(x)\big|} \le r_0(x) \le \Lambda &\ \ale, \\ \label{c2}
		-\frac{\partial}{\partial r} \lambda(x,r) \in \sign\Big(\!\dive\!\big(\kappa(x,\theta_0(x))\!\left(\nabla u_0(x) + \bfnu\right)\big)+v_0(x)-u_0(x)\Big) &\ \ale\ \for r\in (0,r_0(x)), \\ \label{c2a}
		-\kappa(x,\theta_0(x))\big(\nabla u_0(x) + \bfnu\big) \cdot \bfn(x) = b^*(x) (u_0(x) - u^*(x,0))  &\ \ale\ \mbox{\em on }\, \partial\Omega.
	\end{align}
\end{hypothesis}

Unlike \cite{colli}, here we do not need to assume $\!\dive\!\big(\kappa(x,\theta_0(x))\left(\nabla u_0(x) + \bfnu\right)\big) \in L^\infty(\Omega)$ since it follows from the fact that $u_0 \in W^{2,\infty}(\Omega)$ together with assumptions \eqref{ge6}, \eqref{ge3b}, and Hypothesis~\ref{hy2} on $\kappa$. Instead, in addition to the hypotheses of \cite{colli}, we have to assume a polynomial decay of the Preisach density function at infinity that will allow us to perform a hysteresis variant of the Moser iterations in Section~\ref{unif}.

\begin{hypothesis}\label{hrho}
	The Preisach density $\phi :\Omega\times(0,\infty)\times \real \to [0,\infty)$ in \eqref{ge3} is such that
	\begin{equation}\label{c2b}
		\exists m>0\ \exists \phi_0>0: \phi(x,r,v) \ge \phi_0\max\{1,r+|v|\}^{-m} \ \ale\ \for \ (r,v) \in (0,\infty)\times \real.
	\end{equation}
\end{hypothesis}

Condition \eqref{c2b} is indeed compatible with \eqref{irho} if $m>2$, as
$$
\int_0^\infty \int_0^\infty \max\{1,r+v\}^{-m}\dd v\dd r = \int_0^\infty \int_r^\infty \max\{1,z\}^{-m}\dd z\dd r = \frac12 +\frac{1}{m-2},
$$
so that \eqref{irho} is satisfied if for example $\bar G =1/2$ and $\phi_0 \le 1- 2/m$.

Our main existence result reads as follows.

\begin{theorem}\label{t1}
	Let Hypotheses~\ref{hy2}--\ref{hrho} hold, and let $G$ be a convexifiable Preisach operator in the sense of Definition~\ref{dpc}. Then there exists a solution $(u,v)$ to Problem~\eqref{pde}--\eqref{ie3} such that $u,v \in L^\infty(\Omega\times (0,T))$, $\nabla u \in L^2(\Omega\times (0,T);\real^N)$, $u_t$ and $\theta_t = G[u]_t$ belong to the Orlicz space $L^{\Phi_{log}}(\Omega\times (0,T))$ generated by the function $\Phi_{log}(v) = v\log(1+v)$, and $v_t \in L^\infty(\Omega\times (0,T))$.
\end{theorem}

Basic properties of Orlicz spaces are summarized in \cite[Section~5]{perme}. For a more comprehensive discussion, we refer the interested reader to the monographs \cite{ada,rare}.


\section{Time discretization}\label{disc}

We proceed as in \cite{colli}, choose a discretization parameter $n \in \nat$, define the time step $\tau = T/n$, and replace system \eqref{pde}-\eqref{pd0} with the following time-discrete counterpart for the unknowns $\{u_i, v_i: i = 1, \dots, n\}$:
\begin{align}\nonumber
	&\hspace{-9cm}\io \left(\frac1\tau\Big((G[u]_i - G[u]_{i-1}) + (v_i - v_{i-1})\Big)\vrt + \kappa(x,G[u]_{i})\big(\nabla u_i + \bfnu\big)\cdot\nabla\vrt\right)\dd x\\
	\label{dis1}\hspace{17mm} + \ipo b^*(x)(u_i - u^*_i)\vrt \dd s(x) &= 0,\\[2mm] \label{dv1}
	\frac1\tau (v_i - v_{i-1}) + v_i &= u_i,
\end{align}
for every test function $\vrt \in W^{1,2}(\Omega)$, where $u^*_i(x) = u^*(x,i\tau)$ for $i\in \{1,\dots,n\}$, $u_0$ and $v_0$ are as in \eqref{ie3}. Here, the time-discrete Preisach operator $G[u]_i$ is defined for an input sequence $\{u_i : i\in\nat\cup\{0\}\}$ by a formula of the form \eqref{ge3}, namely,
\begin{equation}\label{de3}
	G[u]_i(x) = \bar G + \  \int_0^{\infty}\int_0^{\xi^r_i(x)} \phi(x,r,v)\dd v\dd r,
\end{equation}
where $\xi^r_i$ denotes the output of the time-discrete play operator
\begin{equation}\label{de4}
	\xi^r_i(x) = \play_r[\lambda,u]_i(x)
\end{equation}
defined as the solution operator of the variational inequality
\begin{equation}\label{de4a}
	|u_i(x) - \xi^r_i(x)| \le r, \quad (\xi^r_i(x) - \xi^r_{i-1}(x))(u_i(x) - \xi^r_i(x) - z) \ge 0, \quad \forall i\in \nat, \ \ \forall z \in [-r,r],
\end{equation}
with a given initial condition
\begin{equation}\label{de5}
	\xi^r_0(x) = \lambda(x,r) \ \ale
\end{equation}
similarly as in \eqref{ge4a}-\eqref{ge5}. Note that the discrete variational inequality \eqref{de4a} can be interpreted as weak formulation of \eqref{ge4a} for piecewise constant inputs in terms of the Kurzweil integral, and details can be found in \cite[Section 2]{ele}.

\begin{proposition}
	System \eqref{dis1}-\eqref{dv1} with initial conditions $u_0$ and $v_0$ as in \eqref{ie3} admits at least one solution $\{u_i, v_i\} \subset W^{1,2}(\Omega)$ for each $i\in \{1,\dots,n\}$.
\end{proposition}

\begin{proof}
	We proceed by induction, and assume that the sequences $\{u_i, v_i\}$ have already been constructed for $i = 1,\dots,j-1$. Note that there is no hysteresis in the passage from $u_{j-1}$ to $u_j$, so that there exists a function $G_j: \Omega\times \real \to \real$, continuous and nondecreasing in the second variable, such that $G[u]_j = G_j(x, u_j)$. Furthermore, by \eqref{dv1} we have
	\begin{equation}\label{vj}
		v_j(x) = \frac{1}{1+\tau} v_{j-1}(x) + \frac{\tau}{1+\tau}u_j(x).
	\end{equation}
	Hence, \eqref{dis1}-\eqref{dv1} can be interpreted as a quasilinear elliptic equation for the unknown $u \coloneqq u_j$ of the form
	\begin{equation}\label{req1}
		\begin{aligned}
			&\io \Bigg(\bigg(\frac1\tau G_j(x,u) + \frac1{1+\tau} u\bigg)\vrt + \kappa(x,G_j(x,u))(\nabla u + \bfnu)\cdot\nabla\vrt\Bigg) \dd x\\
			&\hspace{17mm} + \ipo b^*(x)(u - u^*_j)\vrt \dd s(x) = \io a_j y \dd x
		\end{aligned}
	\end{equation}
	for every $y \in W^{1,2}(\Omega)$ with given functions $a_j \in L^2(\Omega)$.
	
	Let $\{e_k: k\in \nat\} \subset W^{1,2}(\Omega)$ be an orthonormal basis of $L^2(\Omega)$. It is convenient to choose eigenfunctions of the following problem, which is compatible with boundary conditions \eqref{bou}, namely
	$$
	\kappa_0 \io \nabla e_k\cdot \nabla y\dd x + \ipo b^*(x) e_k y \dd s(x) = \mu_k \io e_k y\dd x
	$$
	with eigenvalues $0 < \mu_1 <\mu_2 \le \dots$. For each $m \in \nat$ we look for coefficients $u_k$, $k=1, \dots, m$, such that the function
	$$
	u\om(x) = \sum_{k=1}^m u_k e_k(x)
	$$
	satisfies the identity
	\begin{equation}\label{reqm}
		\begin{aligned}
			&\io \Bigg(\bigg(\frac1\tau G_j(x,u\om) + \frac1{1+\tau} u\om\bigg)e_k\dd x+ \kappa(x,G_j(x,u\om))(\nabla u\om + \bfnu)\cdot\nabla e_k\Bigg) \dd x\\ 
			&\hspace{17mm} + \ipo b^*(x)(u\om - u^*_j) e_k \dd s(x) - \io a_j e_k\dd x = 0
		\end{aligned}
	\end{equation}
	for every $k=1, \dots, m$. The existence of a solution to \eqref{reqm} can be proved in a classical way using the Brouwer degree theory. An introduction to topological methods for solving nonlinear partial differential equations can be found in \cite[Chapter~V]{fuku}, but we will mainly refer to \cite{hei}, where detailed proofs are also provided.
	
	We begin by noting that we can define a continuous mapping $T: \real^m \to \real^m$ associating to $\boldsymbol{u} \coloneqq (u_1, \dots, u_m)$ the left-hand side of \eqref{reqm} for $k=1,\dots,m$. Let $\gamma \in [0,1]$, and consider the homotopy
	\begin{equation}\label{reqh}
		\begin{aligned}
			T_\gamma(\boldsymbol{u})_k =& \io \bigg(\frac \gamma\tau G_j(x,u\om) + \frac1{1+\tau} u\om\bigg)e_k \dd x\\
			&   + \io \Big((1-\gamma)\kappa_0 + \gamma\kappa(x,G_j(x,u\om))\Big)(\nabla u\om + \gamma\bfnu)\cdot\nabla e_k \dd x\\ 
			&  + \ipo b^*(x)(u\om - \gamma u^*_j) e_k \dd s(x) - \gamma \io a_j e_k\dd x.
		\end{aligned}
	\end{equation}
	Testing \eqref{reqh} by $u_k$ and summing up over $k=1,\dots,m$, by Hypothesis~\ref{hy2} we obtain
	$$
	\begin{aligned}
		T_\gamma(\boldsymbol{u})\cdot\boldsymbol{u} =& \io \bigg(\frac \gamma\tau G_j(x,u\om)u\om + \frac1{1+\tau}|u\om|^2\bigg) \dd x \\
		&  + \io \Big((1-\gamma)\kappa_0 + \gamma\kappa(x,G_j(x,u\om))\Big)\big(|\nabla u\om|^2 + \gamma\bfnu\cdot\nabla u\om\big) \dd x\\ 
		& + \ipo b^*(x)\big(|u\om|^2 - \gamma u^*_j u\om\big) \dd s(x) - \gamma \io a_j u\om\dd x \\
		&\ge \io \bigg(\frac1{1+\tau}|u\om|^2 + \kappa_0|\nabla u\om|^2\bigg) \dd x + \ipo b^*(x)|u\om|^2 \dd s(x) \\
		&  - \io \bigg(\frac1\tau|G_j(x,u\om)| + |a_j|\bigg)|u\om|\dd x - (\kappa_0+\kappa_1)\io |\nabla u\om| \dd x \\
		&  - \ipo b^*(x)|u_j^*||u\om|\dd s(x).
	\end{aligned}
	$$
	Hence, by Young's inequality and again Hypothesis~\ref{hy2}, we see that $T_\gamma(\boldsymbol{u})\cdot\boldsymbol{u}>0$ outside the ball $B\om_R \subset \real^m$ with a sufficiently large radius $R$ independent of $\gamma$ and $m$. Therefore, the equation $T_\gamma(\boldsymbol{u}) = 0$ has no solution on $\partial B\om_R$, so that by \cite[Theorem~3]{hei} the Brouwer degree $d(T_\gamma,B\om_R,0)$ is a constant for $\gamma \in [0,1]$. The mapping $T_0$ is linear, hence its degree is odd. In particular, $d(T_1,B\om_R,0) = d(T_0,B\om_R,0) \ne 0$, which implies (see \cite[Theorem~2]{hei}) that the equation $T_1(\boldsymbol{u}) = 0$, equivalently, \eqref{reqm}, has at least one solution in $B\om_R$.
	
	Testing \eqref{reqm} by $u_k$, summing up over $k=1,\dots,m$, and employing Hypothesis~\ref{hy2} we see that the sequence $\{u\om\}$ is uniformly bounded in $W^{1,2}(\Omega)$, hence it is compact in $L^2(\Omega)$. Let $u$ be the limit of any convergent subsequence of $\{u\om\}$. Passing to the limit in \eqref{reqm} as $m\to \infty$, we check that \eqref{req1} holds with $y=e_k$ for every $k \in \nat$. Since the system $\{e_k\}$ is complete in $L^2(\Omega)$, \eqref{req1} holds for $u = u_j$, and recalling \eqref{vj} we conclude that system~\eqref{dis1}-\eqref{dv1} has a solution.
\end{proof}


\section{Uniform upper bounds}\label{unif}

The situation is different from the case without gravity studied in \cite{perme}, where the upper bound could be derived in an elementary way from Hilpert's inequality following \cite{hilp}. Here, we propose a hysteresis variant of the Moser iteration procedure which is new to our knowledge. The price we pay is that we have to restrict the class of admissible Preisach densities $\phi(x,r,v)$ assuming that they have a sufficiently slow decay at infinity, see Hypothesis~\ref{hrho}.

We first test \eqref{dis1} by $\vrt = u_i$ and get, using \eqref{dv1},
\begin{align}\nonumber
	&\io \left(\frac1\tau\Big((G[u]_i - G[u]_{i-1})u_i + (v_i - v_{i-1})v_i\Big) + \kappa(x,G[u]_{i})\big(\nabla u_i + \bfnu\big)\cdot\nabla u_i\right)\dd x\\
	\label{des2}&\hspace{17mm}  +\frac1{\tau^2}\io |v_i - v_{i-1}|^2\dd x + \ipo b^*(x)(u_i - u^*_i) u_i \dd s(x) = 0
\end{align}
for all $i\in\{1,\dots,n\}$. Let us define the functions
\begin{equation}\label{psi}
	\psi(x,r,\xi) \coloneqq \int_0^\xi\phi(x,r,v)\dd v, \quad \Psi (x,r,\xi) \coloneqq \int_0^\xi v \phi(x,r,v)\dd v.
\end{equation}
In terms of the sequence $\xi^r_i(x) = \play_r[\lambda,u]_i$ we have
\begin{equation}\label{Gi}
	G[u]_i(x) = \bar G + \int_0^\infty \psi(x,r,\xi^r_i(x))\dd r.
\end{equation}
Choosing in \eqref{de4a} $z = 0$ and using the fact that the function $\psi$ in \eqref{psi} is increasing, we obtain in both cases $\xi^r_i\ge \xi^r_{i-1}$, $\xi^r_i\le \xi^r_{i-1}$ the inequalities
\begin{equation}\label{psi1}
	(\psi(x,r,\xi^r_i)-\psi(x,r,\xi^r_{i-1})) u_i \ge (\psi(x,r,\xi^r_i)-\psi(x,r,\xi^r_{i-1}))\xi^r_i \ge \Psi(x,r,\xi^r_i)-\Psi(x,r,\xi^r_{i-1}),
\end{equation}
and \eqref{des2} yields, by Hypothesis~\ref{hy2} and the inequality $(v_i - v_{i-1}) v_i \ge (v_i^2 - v_{i-1}^2)/2$,
\begin{equation}\label{energ}
	\begin{aligned}
		&\frac1\tau \io\bigg(\int_0^{\infty} \big(\Psi(x,r,\xi^r_i)-\Psi(x,r,\xi^r_{i-1})\big)\dd r + \frac{v_i^2-v_{i-1}^2}{2}\bigg) \dd x \\
		&+\kappa_0\io |\nabla u_i|^2\dd x + \ipo b^*(x)|u_i|^2 \dd s(x) \le C
	\end{aligned}
\end{equation}
for $i \in \{1,\dots,n\}$, with a constant $C>0$ independent of $\tau$. Summing up over $i$ and exploiting the fact that, by the definition of $\xi^r_0$ in \eqref{de5} and the assumptions on $\phi$ and $\lambda$ in Definition~\ref{dpr}, we have
\begin{equation}\label{enini}
	\int_0^{\infty} \Psi(x,r,\xi^r_0(x)) \dd r = \int_0^{\Lambda}\int_0^{\lambda(x,r)} v\phi(x,r,v) \dd v\dd r \le \frac{\phi_1}{2} \int_0^\Lambda \lambda^2(x,r) \dd r \le C,
\end{equation}
and we are assuming $v_0 \in L^\infty(\Omega)$, we get the estimate
\begin{equation}\label{energy}
	\max_{i=1,\dots,n}\io |v_i|^2 \dd x + \tau\sumiz \left(\io |\nabla u_i|^2\dd x + \ipo b^*(x)|u_i|^2 \dd s(x)\right) \le C
\end{equation}
with a constant $C>0$ independent of $\tau$.

For parameters $R>1$ and $k>1$ to be specified later, we define the function
\begin{equation}\label{fpr}
	\fpr(u) \coloneqq \left\{
	\begin{array}{ll}
		u |u|^{2k} & \for \ |u| \le R,\\[2mm]
		(2k+1)R^{2k} u - 2kR^{2k+1} & \for \ u > R,\\[2mm]
		(2k+1)R^{2k} u + 2kR^{2k+1} & \for \ u < -R.
	\end{array}\right.
\end{equation}
The function $\fpr$ is odd, increasing, continuously differentiable, and its derivative has the form
\begin{equation}\label{fpr2}
	\fpr'(u) = \left\{
	\begin{array}{ll}
		(2k+1) |u|^{2k} & \for \ |u| \le R,\\[2mm]
		(2k+1) R^{2k} & \for \ |u|> R.
	\end{array}\right.
\end{equation}
We test \eqref{dis1} with $\vrt = \fpr(u_i)$, which is an admissible test function, and we obtain
\begin{align}\nonumber
	&\io \left(\frac1\tau(G[u]_i {-} G[u]_{i-1}) \fpr(u_i) + (v_i {-} v_{i-1})\fpr(u_i) + \fpr'(u_i)\kappa(x,G[u]_{i})\big(\nabla u_i {+} \bfnu\big)\cdot\nabla u_i \right)\dd x\\
	\label{desp}&\hspace{17mm} + \ipo b^*(x)(u_i - u^*_i) \fpr(u_i) \dd s(x) = 0.
\end{align}
For every $\alpha, \beta \in \real$ and for every nondecreasing function $h:\real \to \real$, we have the elementary inequality
\begin{equation}\label{elem}
	\alpha \big(h(\alpha+\beta) - h(\beta)\big) \ge 0.
\end{equation}
Formula \eqref{elem} for $\alpha = (v_i - v_{i-1})/\tau$, $\beta = v_i$, $h = \fpr$ together with \eqref{dv1} enables us to estimate
\begin{equation}\label{esvi}
	(v_i - v_{i-1})\fpr(u_i) \ge (v_i - v_{i-1})\fpr(v_i) \ge \hukr(v_i) - \hukr(v_{i-1}),
\end{equation}
where we put
\begin{equation}\label{huk}
	\hukr(v) = \int_0^v \fpr(z)\dd z \ge 0.
\end{equation}
The next term in \eqref{desp} can be estimate from below by means of Young's inequality
\begin{equation}\label{esi1}
	\fpr'(u_i)\kappa(x,G[u]_{i})\big(\nabla u_i + \bfnu\big)\cdot\nabla u_i \ge a \fpr'(u_i)|\nabla u_i|^2 - b \fpr'(u_i)
\end{equation}
with positive constants $a,b$ independent of $R$ and $k$. To estimate the boundary term in \eqref{desp} from below, we first notice that for all $u\in \real$ we have, by virtue of \eqref{fpr},
\begin{equation}\label{esi2}
	\frac{|\fpr(u)|^{(2k+2)/(2k+1)}}{u \fpr(u)} = \frac{|\fpr(u)|^{1/(2k+1)}}{|u|} \le 1.
\end{equation}
Since $u^*$ is bounded by Hypothesis~\ref{hy2}, we may use \eqref{esi2} and Young's inequality with conjugate exponents $2k+2$ and $(2k+2)/(2k+1)$ to check that there exists a constant $C>0$ independent of $R$ and $k$ such that
\begin{equation}\label{esi3}
	(u_i - u^*_i) \fpr(u_i) \ge \frac12 u_i \fpr(u_i) - C^{2k+2}
\end{equation}
for all $i\in\{1,\dots,n\}$.

Finally, in order to estimate the time-discrete hysteresis term in \eqref{desp}, we define the function
\begin{equation}\label{psip}
	\Psi_{k,R} (x,r,\xi) \coloneqq \int_0^\xi \fpr(v) \phi(x,r,v)\dd v.
\end{equation}
Similarly as in \eqref{psi1}, we have
\begin{align}\nonumber
	(\psi(x,r,\xi^r_i)-\psi(x,r,\xi^r_{i-1})) \fpr(u_i) &\ge (\psi(x,r,\xi^r_i)-\psi(x,r,\xi^r_{i-1}))\fpr(\xi^r_i)\\ \label{psip1}
	&\ge \Psi_{k,R}(x,r,\xi^r_i)-\Psi_{k,R}(x,r,\xi^r_{i-1}).
\end{align}
Summing up in \eqref{desp} over $i = 1, \dots, n$ and using the above estimates \eqref{esvi}, \eqref{esi1}, \eqref{esi3}, and~\eqref{psip1}, we get, using again an estimate of the initial step at $i=0$ similar to \eqref{enini} and the assumption $v_0 \in L^\infty(\Omega)$,
\begin{align}\nonumber
	& \max_{i=1,\dots,n}\io\int_0^\infty \Psi_{k,R}(x,r,\xi^r_i)\dd r\dd x + a\tau\sumin\Big(\io \fpr'(u_i)|\nabla u_i|^2\dd x + \ipo b^*(x)u_i \fpr(u_i)\dd s(x)\Big)\\ \label{esi4}
	& \qquad\qquad \le C\Big(C^{2k} + \tau\sumin\io \fpr'(u_i) \dd x\Big)
\end{align}
with constants $a>0$ and $C>0$ independent of $\tau$, $k$, and $R$.

For $u\in \real$ put
\begin{equation}\label{hatf}
	\fhr(u)\coloneqq \left\{
	\begin{array}{ll}
		u |u|^{k} & \for \ |u| \le R,\\[2mm]
		(k+1)R^{k} u - kR^{k+1} & \for \ u > R,\\[2mm]
		(k+1)R^{k} u + kR^{k+1} & \for \ u < -R.
	\end{array}\right.
\end{equation}
We have
\begin{align}
	|\nabla \fhr(u_i)|^2 &= \left\{
	\begin{array}{ll}
		(k+1)^2 |u_i|^{2k} |\nabla u_i|^2 & \for \ |u_i|<R,\\[2mm]
		(k+1)^2 R^{2k} |\nabla u_i|^2 & \for \ |u_i| \ge R,
	\end{array}\right. \label{nabVkR}\\[2mm]
	|\fhr(u_i)|^2 &= \left\{
	\begin{array}{ll}
		|u_i|^{2k+2}  & \for \ |u_i|<R,\\[2mm]
		(k+1)^2 R^{2k}|u_i|^2 + k^2 R^{2k+2} - 2k(k+1)R^{2k+1}|u_i| \ge R^{2k+2} & \for \ |u_i| \ge R.
	\end{array}\right. \label{VkR}
\end{align}
Then
\begin{align}
	\fpr'(u_i)|\nabla u_i|^2 &= \frac{2k+1}{(k+1)^2}\big|\nabla \fhr(u_i)\big|^2, \label{UV1} \\[2mm]
	u_i \fpr(u_i) &\ge \frac{2k+1}{(k+1)^2} \big|\fhr(u_i)\big|^2. \label{UV2}
\end{align}
We define an equivalent norm of an element $w$ from the Sobolev space $W^{1,2}(\Omega)$ by the formula
$$
\|w\|_{1,2} \coloneqq \Big(a\io |\nabla w|^2\dd x + \ipo b^*(x)|w|^2\dd s(x)\Big)^{1/2}
$$
and using \eqref{UV1}-\eqref{UV2} rewrite inequality \eqref{esi4} in the form
\begin{equation}\label{esi6}
	\max_{i=1,\dots,n}\io\int_0^\infty \Psi_{k,R}(x,r,\xi^r_i)\dd r\dd x + \frac{\tau}{k+1}\sumin \|\fhr(u_i)\|_{1,2}^2\le C\Big(C^{2k} + \tau\sumin\io \fpr'(u_i) \dd x\Big).
\end{equation}
We define
\begin{equation}\label{ui}
	w_i(x) \coloneqq \min\{R, |u_i(x)|\}, \quad \uipr(x) \coloneqq w_i^k(x).
\end{equation}

With this notation, by \eqref{fpr2} we can rewrite the term at the right-hand side of \eqref{esi6} as
\begin{equation}\label{fpri}
	\io \fpr'(u_i) \dd x = (2k+1)\io |\uipr(x)|^2 \dd x.
\end{equation}
Let us now compare the $W^{1,2}$-norm of $\uipr$ with the $W^{1,2}$-norm of $V_{k,R}(u_i)$ on the left-hand side 
of~\eqref{esi6}. We have
$$
|\nabla \uipr|^2 = \left\{
\begin{array}{ll}
	k^2 |u_i|^{2k-2} |\nabla u_i|^2 & \for \ |u_i|<R\\[2mm]
	0 & \for \ |u_i| \ge R,
\end{array}\right.
\qquad
|\uipr|^2 = \left\{
\begin{array}{ll}
	|u_i|^{2k}  & \for \ |u_i|<R\\[2mm]
	R^{2k} & \for \ |u_i| \ge R,
\end{array}\right.
$$
and, by Young's inequality with conjugate exponents $k$ and $k/(k-1)$,
$$
k^2 |u_i|^{2k-2} \le k^2\left(\frac1k + \frac{k-1}k |u_i|^{2k}\right) \le k + (k+1)^2|u_i|^{2k}.
$$
Hence, recalling \eqref{nabVkR}-\eqref{VkR}, we estimate
\begin{equation}\label{w12}
	\|\uipr\|_{1,2}^2 \le \|\fhr(u_i)\|_{1,2}^2+ k \|u_i\|_{1,2}^2
\end{equation}
for $i=1, \dots, n$. From \eqref{esi6}, \eqref{fpri}, and \eqref{energy} we thus get
\begin{equation}\label{esi7}
	\max_{i=1,\dots,n}\io\int_0^\infty \Psi_{k,R}(x,r,\xi^r_i)\dd r\dd x + \frac{\tau}{k+1}\sumin \|\uipr\|_{1,2}^2\le C_0\Big(C^{2k} + (k+1)\tau\sumin\io \big|\uipr\big|^2 \dd x\Big),
\end{equation}
for some constants $C_0,C>0$ independent of $\tau$, $k$, and $R$. To simplify the notation, we denote by $|\cdot|_q$ the norm in $L^q(\Omega)$. There exists a constant $K>0$ such that for every $\sigma\in (0,1)$ the interpolation inequality
\begin{equation}\label{ipol}
	|w|_2^2 \le K\big(\sigma^{-N}|w|_1^2 + \sigma^2 \|w\|_{1,2}^2\big)
\end{equation}
holds for every $w \in W^{1,2}(\Omega)$, see \cite{bin}. Choosing $\sigma = (C_0 K)^{-1/2}(1+k)^{-1}$, we apply this inequality in~\eqref{esi7} and obtain
\begin{equation}\label{moser1}
	\max_{i=1,\dots,n}\io\int_0^\infty \Psi_{k,R}(x,r,\xi^r_i)\dd r\dd x\le C\left(C^{2k} + (k+1)^{N+1}\,\tau \sumin\left(\io \big|\uipr\big| \dd x\right)^2\right),
\end{equation}
with a constant $C>0$ depending on $C_0$ and $K$.

So far, this has been a standard Moser argument. The crucial point in the derivation of an upper bound for the time-discrete problem is to find an efficient lower bound for the hysteresis term on the left-hand side of \eqref{moser1}. By \eqref{psip} we have for $x \in \Omega$ that
\begin{equation}\label{intpsi1}
	\int_0^\infty\Psi_{k,R}(x,r,\xi^r_i(x))\dd r = \int_0^\infty\int_0^{\xi^r_i(x)}\fpr(v) \phi(x,r,v)\dd v \dd r,
\end{equation}
where the function $\fpr$ is defined in \eqref{fpr}. Assume first that $u_i(x) > 0$. By \eqref{de4a} for all $r>0$ we have
$$
\int_0^{\xi^r_i(x)}\fpr(v) \phi(x,r,v)\dd v \ge \int_0^{(\xi^r_i(x))^+}\fpr(v) \phi(x,r,v)\dd v \ge \int_0^{(u_i(x) - r)^+}\fpr(v) \phi(x,r,v)\dd v,
$$
where $(\cdot)^+$ denotes the positive part. From \eqref{intpsi1} we thus get
\begin{equation}\label{intpsi2}
	\int_0^\infty\Psi_{k,R}(x,r,\xi^r_i(x))\dd r \ge \int_0^{u_i(x)}\int_0^{u_i(x) - r}\fpr(v) \phi(x,r,v)\dd v \dd r.
\end{equation}
In the case $u_i(x) < 0$ we argue similarly and obtain the formula which is valid for both cases
\begin{align}\nonumber
	\int_0^\infty\Psi_{k,R}(x,r,\xi^r_i(x))\dd r & \ge \int_0^{|u_i(x)|}\int_0^{|u_i(x)| - r}\fpr(v) \phi(x,r,v)\dd v \dd r\\[2mm] \label{intpsi3}
	&  \ge \phi_0 \int_0^{|u_i(x)|}\int_0^{|u_i(x)| - r}\fpr(v)\max\{1, r+v\}^{-m}\dd v\dd r \eqqcolon \phi_0 A_{k,R}(|u_i(x)|),
\end{align}
according to Hypothesis~\ref{hrho}. Our goal is to show that the function
\begin{equation}\label{AA}
	\begin{aligned}
		A_{k,R}(u) &= \int_0^{u}\int_0^{u - r}\max\{1, r+v\}^{-m}\fpr(v)\dd v\dd r\\
		& = \int_0^u\int_0^{u-v} \max\{1, r+v\}^{-m}\fpr(v)\dd r\dd v \\
		&= \int_0^u\fpr(v) \int_v^u\max\{1, z\}^{-m}\dd z\dd v,
	\end{aligned}
\end{equation}
of the argument $u\ge 0$, where we have first used Fubini's theorem and then substituted $z=r+v$, is dominant for $k > (m-3)/2$ over the function
$$
B_{k,R}(u) \coloneqq \min\{R,u\}^{2k+3 - m},
$$
that is, there exists a constant  $C_k$ independent of $R$ and possibly dependent on $k$ such that
\begin{equation}\label{AB}
	B_{k,R}(u) \le 1 + C_k A_{k,R}(u)
\end{equation}
for all $u \ge 0$. We have in particular
\begin{align*}
	A'_{k,R}(u) &= \max\{1, u\}^{-m}\int_0^u\fpr(v)\dd v, \\[2mm]
	B'_{k,R}(u) &= (2k+3-m) u^{2k+2-m} \ \for \ u<R, \quad B'_{k,R}(u) = 0\  \for u>R.
\end{align*}
We obviously have \eqref{AB} for $u \le 1$. Indeed, in this case $u \le 1 < R$, and we can compute
\begin{align*}
	A_{k,R}(u) &= \int_0^u \fpr(v)(u-v) \dd v = \frac{u^{2k+3}}{(2k+2)(2k+3)}, \\
	B_{k,R}(u) &= u^{2k+3-m},
\end{align*}
so that \eqref{AB} holds with $C_k = 6(k+1)^2$. For $u > 1$ it suffices to prove that
\begin{equation}\label{ABd}
	B'_{k,R}(u) \le C_k A'_{k,R}(u).
\end{equation}
Indeed, then \eqref{AB} follows from the formula $B_{k,R}(u) - B_{k,R}(1) \le C_k (A_{k,R}(u) - A_{k,R}(1))$ and the fact that~\eqref{AB} holds for $u=1$ from the previous step. Note that \eqref{ABd} holds automatically for $u>R$. For $u \in (1,R)$ we have by \eqref{fpr} that
$$
A'_{k,R}(u) = u^{-m} \int_0^u v^{2k+1}\dd v = \frac1{2k+2} u^{2k+2-m},
$$
hence \eqref{ABd} and \eqref{AB} hold with $C_k$ as above. Therefore, combining \eqref{intpsi3} and \eqref{AB}, in terms of the functions $w_i(x)$ defined in \eqref{ui}, estimate \eqref{moser1} implies
\begin{equation}\label{moser3}
	\max_{i=1,\dots,n}\left|w_i\right|_{2k+3-m}^{2k+3-m} \le C k^{N+3}\max\left\{C^{2k}, \tau\sumin\left|w_i\right|_{k}^{2k} \right\}
\end{equation}
for all $k > (m-3)/2$ with a constant $C>0$ independent of $k$ and $R$.

We are ready to start the Moser iterations. Note that we can assume $m\ge3$. Indeed, if \eqref{c2b} holds with $m<3$, it will certainly be true with $m=3$. Since we have an anisotropic norm on the right-hand side of \eqref{moser3}, we have to replace the left-hand side by an expression which is compatible with the right-hand side. We shall see that, since
$$
\left(\tau\sumin\left|w_i\right|_{2k+3-m}^{4k+6-2m}\right)^{1/2} \le \left(\tau n \max_{i=1,\dots,n} \Big(\left|w_i\right|_{2k+3-m}^{2k+3-m}\Big)^2\right)^{1/2},
$$
the right choice is
\begin{equation}\label{moser4}
	\max\left\{C^{2k+3-m}, \left(\tau\sumin\left|w_i\right|_{2k+3-m}^{4k+6-2m}\right)^{1/2}\right\} \le C_T k^{N+3}\max\left\{C^{2k}, \tau\sumin\left|w_i\right|_{k}^{2k} \right\}
\end{equation}
with a constant $C_T$ depending on $T$. We now choose $k_0> m-3$ and define a sequence $k_j$ for $j \in \nat$ recurrently by the formula $k_{j} = 2k_{j-1} + 3 - m$, that is,
\begin{equation}\label{kj}
	k_j = 2^j K_m +m - 3, \quad K_m = k_0 +3 -m.
\end{equation}
From \eqref{moser4} we obtain
\begin{equation}\label{moser5}
	\left(\max\left\{C, \left(\tau\sumin\left|w_i\right|_{k_j}^{2k_j}\right)^{1/2k_j}\right\}\right)^{\omega_j} \le \left(C_T k_{j-1}^{N+3}\right)^{1/2k_{j-1}}\max\left\{C, \left(\tau\sumin\left|w_i\right|_{k_{j-1}}^{2k_{j-1}}\right)^{1/2 k_{j-1}} \right\}
\end{equation}
with exponent
$$
\omega_j = \frac{2k_{j-1}+3-m}{2k_{j-1}} = 1- \frac{m-3}{2k_{j-1}}.
$$
Note that $(m-3)/2k_{j-1} \in [0,1/2)$, so that $\omega_j \in (1/2,1]$. We define the quantities
$$
L_j \coloneqq \log\left(\max\left\{C, \left(\tau\sumin\left|w_i\right|_{k_j}^{2k_j}\right)^{1/2k_j} \right\}\right).
$$
Then from \eqref{moser5} we get the recurrent relation
\begin{equation}\label{ome}
	\omega_j L_{j} \le \frac1{2k_{j-1}} (\log C_T + (N+3) \log k_{j-1}) + L_{j-1}
\end{equation}
for all $j \in \nat$. We rewrite \eqref{ome} in the form
\begin{equation}\label{ome2}
	\left(\prod_{i=1}^j \omega_i\right) L_{j} \le \left(\prod_{i=1}^{j-1}\omega_i\right) \frac1{2k_{j-1}} (\log C_T + (N+3) \log k_{j-1}) + \left(\prod_{i=1}^{j-1}\omega_i\right) L_{j-1}.
\end{equation}
Summing up the above inequality over $j=2,\dots,k$ yields
\begin{equation}\label{ome3}
	L_k \le \left(\prod_{i=1}^{k} \frac{1}{\omega_i}\right) \sum_{j=2}^k \left(\prod_{i=1}^{j-1}\omega_i\right) \frac1{2k_{j-1}} (\log C_T + (N+3) \log k_{j-1}) + \left(\prod_{i=2}^{k} \frac{1}{\omega_i}\right) L_1.
\end{equation}
To estimate $L_1$, we observe that by virtue of \eqref{moser4} we have
\begin{equation}\label{moser6}
	\begin{aligned}
		\left(\tau\sumin\left|w_i\right|_{k_1}^{2k_1}\right)^{1/2k_1} &= \left(\left(\tau\sumin\left|w_i\right|_{2k_0+3-m}^{4k_0+6-2m}\right)^{1/2}\right)^{1/(2k_0+3-m)} \\
		&\le C\left(1 + \tau\sumin\left|w_i\right|_{k_0}^{2k_0}\right)^{1/(2k_0+3-m)}.
	\end{aligned}
\end{equation}
Poincar\'{e}'s inequality implies
\begin{equation}\label{vrk}
	|\fhr(u_i)|_2^2 \le C \|\fhr(u_i)\|_{1,2}^2,
\end{equation}
so that from \eqref{esi6} and \eqref{fpri} it follows that for all $k\ge 1$ we have
\begin{equation}\label{esi6a}
	\begin{aligned}
		\tau\sumin |w_i|_{2k+2}^{2k+2} &\le \tau\sumin|\fhr(u_i)|_2^2 \\
		&\le C \tau\sumin\|\fhr(u_i)\|_{1,2}^2 \\
		&\le C(k+1)\Big(C^{2k} + (2k+1)\tau\sumin |w_i|_{2k}^{2k}\Big).
	\end{aligned} 
\end{equation}
By induction over $k$ we prove that for $k_0 > m-3$ as above the estimate
\begin{equation}\label{esi6b}
	\tau\sumin |w_i|_{2k_0}^{2k_0} \le C
\end{equation}
holds true with a constant $C$ depending only on $m$ in Hypothesis~\ref{hrho} and on the data of the problem. Hence, by \eqref{moser6} and \eqref{esi6b} we obtain that $L_1 \le C$. Coming back to \eqref{ome3}, we note also that
$$
\frac{1}{\omega_i} = 1 + \frac{m-3}{2k_{i-1}+3-m},
$$
and since
$$
\prod_{i=1}^\infty \frac{1}{\omega_i} < \infty \ \Longleftrightarrow \ \sum_{i=1}^\infty \log\bigg(\frac{1}{\omega_i}\bigg) < \infty \ \Longleftrightarrow \ \sum_{i=1}^\infty \frac{m-3}{2k_{i-1}+3-m} < \infty,
$$
we conclude that $\prod_{i=1}^\infty 1/\omega_i < \infty$. We additionally have
$$
\prod_{i=1}^\infty \omega_i > 0 \ \Longleftrightarrow \ \sum_{i=1}^\infty \frac{m-3}{2k_{i-1}} < \infty,
$$
see \cite[Theorem~15.5]{rudin}, so that $\prod_{i=1}^\infty \omega_i > 0$. To summarize, we have
$$
\sum_{j=1}^\infty \frac1{2k_{j-1}} (\log C_T + (N+3) \log k_{j-1}) < \infty, \quad \prod_{i=1}^\infty\frac{1}{\omega_i} < \infty, \quad \prod_{i=1}^\infty\omega_i > 0,
$$
and we conclude from \eqref{ome3} that there exists a constant $L^*>0$ such that $L_k \le L^*$ for all $k\in \nat$, hence
\begin{equation}\label{Lam}
	|u_i(x)| \le U \coloneqq \expe^{L^*}
\end{equation}
for a.e.\ $x \in \Omega$ and all $i\in \{1,\dots,n\}$. By comparison in Eq \eqref{dv1}, we also have
\begin{equation}\label{vLam}
	|v_i(x)| \le C
\end{equation}
for a.e.\ $x \in \Omega$ and all $i\in \{1,\dots,n\}$, with a constant $C>0$ depending on $U$. Estimate \eqref{Lam} additionally implies $|\xi_i^r(x)| \le (\hat{U}-r)^+$ for $\hat{U}\coloneqq\max\{U,\Lambda\}$, and we can write in \eqref{Gi}
$$
G[u]_i(x) = \bar G + \int_0^{\hat{U}} \psi(x,r,\xi^r_i(x))\dd r.
$$
Hence, even if we do not assume the a priori boundedness of $G$ as in \eqref{irho}, by assumption \eqref{ge3a} we obtain
\begin{equation}\label{GiL}
	\quad |G[u]_i| \le C
\end{equation}
with a constant $C>0$ independent of $i$ and $\tau$.


\section{Convexity estimate}\label{conv}

Recall that the operator $G$ is convexifiable in the sense of Definition~\ref{dpc}, that is, for every $U>0$ there exists a twice continuously differentiable mapping $g:[-U, U] \to [-U, U]$ such that $g(0) = 0$, $0 < g_* \le g'(u) \le g^* < \infty$, $|g''(U)| \le \bar{g}$, and $G$ is of the form
\begin{equation}\label{ne0}
	G = P \circ g,
\end{equation}
where $P$ is a uniformly counterclockwise convex Preisach operator on $[-U,U]$. Let us fix $U$ from~\eqref{Lam} and the corresponding function $g$. The following result is a straightforward consequence of \cite[Proposition~3.6]{colli}.

\begin{proposition}\label{pc}
	Let $P$ be uniformly counterclockwise convex on $[-U,U]$, and let $f$ be an odd increasing function such that $f(0) = 0$. Then there exists $\beta>0$ such that for every sequence $\{w_i: i = -1, 0, \dots, n-1\}$ in $[-U,U]$ we have
	\begin{equation}\label{Pi}
		\begin{aligned}
			&\sumim (P[w]_{i+1} - 2P[w]_i + P[w]_{i-1})f(w_{i+1} - w_i) + \frac{P[w]_0 - P[w]_{-1}}{w_0 - w_{-1}}F(w_0 - w_{-1}) \\[2mm]
			&\qquad \ge \frac{\beta}{2}\sumim \Gamma(w_{i+1} - w_i),
		\end{aligned}
	\end{equation}
	where we set for $w\in \real$
	\begin{equation}\label{gf1}
		F(w) \coloneqq \int_0^w f(v)\dd v, \qquad \Gamma(w) \coloneqq |w|(wf(w) - F(w)) = |w|\int_0^{|w|}vf'(v)\dd v.
	\end{equation}
\end{proposition}

We need to define backward steps $u_{-1}$ and $v_{-1}$ satisfying the strong formulation of \eqref{dis1}-\eqref{dv1} for $i=0$, that is,
\begin{align}
	\frac1\tau\Big((G[u]_0(x) - G[u]_{-1}(x)) +(v_0(x)- v_{-1}(x))\Big) &= \dive\!\big(\kappa(x,\theta_0(x))\left(\nabla u_0(x) + \bfnu\right)\big) \ \mbox{ in } \Omega, \label{e7} \\
	v_{-1}(x) &= (1+\tau)v_0(x) - \tau u_0(x) \ \mbox{ in } \Omega, \label{e7a}
\end{align}
with boundary condition \eqref{c2a}.
Repeating the argument of \cite[Proposition~3.3]{colli}, we use assumptions~\eqref{ge3a} and \eqref{c2} to find for each $0<\tau <\phi_0(U)/2L^2$ functions $u_{-1}$ and $G[u]_{-1}$ satisfying
\begin{equation}\label{e7b}
	\frac1\tau(G[u]_0(x) - G[u]_{-1}(x)) = \dive\!\big(\kappa(x,\theta_0(x))\left(\nabla u_0(x) + \bfnu\right)\big) + v_0(x) - u_0(x) \ \mbox{ in } \Omega,
\end{equation}
as well as, thanks to \eqref{c1}, the estimate
\begin{equation}\label{inim}
	\frac1\tau |u_0(x) - u_{-1}(x)| \le C
\end{equation}
with a constant $C>0$ independent of $\tau$ and $x$. Finally, we construct $v_{-1}$ according to \eqref{e7a}.

We extend the discrete system \eqref{dis1}-\eqref{dv1} to $i=0$ and write it in the form
\begin{align}\nonumber
	&  \io \left(\frac1\tau\Big((P[w]_i - P[w]_{i-1})  + (v_i - v_{i-1})\Big)\vrt + \kappa(x,\theta_{i})\big(\nabla u_i  + \bfnu\big)\cdot\nabla\vrt\right)\dd x\\
	\label{dp1}& \hspace{2.5cm} + \ipo b^*(x)(u_i - u^*_i)\vrt \dd s(x) = 0,\\[2mm] \label{dvp1}
	&\hspace{3cm} \frac1\tau (v_i - v_{i-1}) + v_i = u_i,
\end{align}
with $w_i = g(u_i)$, $\theta_i = G[u]_i$, for $i\in \{0,1,\dots,n\}$ and for an arbitrary test function $\vrt \in W^{1,2}(\Omega)$. We proceed as in \cite{colli} and test the difference of \eqref{dp1} taken at discrete times $i+1$ and $i$
\begin{equation}\label{dp2}
	\begin{aligned} 
		&\io \frac1\tau\big(P[w]_{i{+}1}  - 2P[w]_i  +  P[w]_{i{-}1} +v_{i{+}1} - 2v_i  +  v_{i{-}1}\big)\vrt \dd x
		\\ &\qquad + \io \Big(\kappa(x,\theta_{i{+}1})\nabla u_{i{+}1} - \kappa(x,\theta_i) \nabla u_i  + (\kappa(x,\theta_{i{+}1}) - \kappa(x,\theta_{i}))\bfnu\Big) \cdot\nabla\vrt\dd x\\
		&\qquad + \ipo b^*(x)(u_{i+1} -u_i)\vrt \dd s(x) = \ipo b^*(x)(u^*_{i+1} - u^*_i)\vrt \dd s(x)
	\end{aligned}
\end{equation}
by $\vrt = f(w_{i+1} - w_i)$ with 
\begin{equation}\label{dp2f}
	f(w) \coloneqq  \frac{w}{\tau + |w|}.
\end{equation}
In agreement with \eqref{gf1}, we have
\begin{align}\label{ff}
	F(w) &= |w| - \tau\log\left(1 + \frac{|w|}{\tau}\right), \\ \label{fgamma}
	\Gamma(w) &= \tau |w|\left(\log\left(1 + \frac{|w|}{\tau}\right) - \frac{|w|}{\tau + |w|}\right).
\end{align}
The hysteresis term is estimated from below by virtue of Proposition~\ref{pc} as follows:
\begin{equation}\label{dp3}
	\begin{aligned}
		\frac1\tau \sumim &(P[w]_{i+1} - 2P[w]_i + P[w]_{i-1})f(w_{i+1} - w_i) + \frac{1}{\tau}\frac{P[w]_0-P[w]_{-1}}{w_0-w_{-1}}F(w_0-w_{-1}) \\
		&\ge \frac{\beta}{2\tau}\sumim \Gamma(w_{i+1} - w_i).
	\end{aligned}
\end{equation}
Note that for $|w| \ge \tau(\expe^2 - 1)$, we have
$$
\frac{|w|}{\tau + |w|} < 1 \le \frac12\log\left(1 + \frac{|w|}{\tau}\right),
$$
so that
$$
\Gamma(w) \ge \frac{\tau}{2}|w|\log\left(1 + \frac{|w|}{\tau}\right).
$$
We denote by $J$ the set of all $i \in \{0,1,\dots, n-1\}$ such that $|w_{i+1} - w_i| \ge \tau(\expe^2 - 1)$, and by $J^\perp$ its complement. We thus have
\begin{align}\label{dp3a}
	\frac12\sum_{i\in J} |w_{i+1} - w_i|\log\left(1 + \frac{|w_{i+1} {-} w_i|}{\tau}\right) &\le \frac1\tau\sum_{i\in J}\Gamma(w_{i+1} - w_i), \\ \label{dp3c}
	\frac12\sum_{i\in J^\perp} |w_{i+1} - w_i|\log\left(1 + \frac{|w_{i+1} {-} w_i|}{\tau}\right) &\le T(\expe^2 - 1),
\end{align}
hence
$$
\frac{\beta}{2\tau}\sumim \Gamma(w_{i+1} - w_i) \ge \frac\beta{4}\sumim |w_{i+1} - w_i|\log\left(1 + \frac{|w_{i+1} {-} w_i|}{\tau}\right) - C
$$
with a constant $C>0$ independent of $x$ and $\tau$. Moreover, for $w_0 \neq w_{-1}$ we have
$$
0 < \frac{F(w_0-w_{-1})}{|w_0-w_{-1}|} \le 1,
$$
which yields, together with identity \eqref{e7b} and assumption \eqref{c1},
\begin{equation}\label{dp3b}
	0 < \frac{1}{\tau}\left|\frac{P[w]_0-P[w]_{-1}}{w_0-w_{-1}}\right| F(w_0-w_{-1}) \le C
\end{equation}
with a constant $C>0$ independent of $x$ and $\tau$. For $w_0 = w_{-1}$, we interpret $(P[w]_0-P[w]_{-1})/(w_0-w_{-1})$ as $B'_+(w_{-1})$ or $B'_-(w_{-1})$, according to the notation for the Preisach branches introduced after Proposition~\ref{pc1}, see \cite{colli} for more details.
From \eqref{dp3}--\eqref{dp3b} we thus get
\begin{equation}\label{dp4}
	\frac1\tau \sumim (P[w]_{i+1} - 2P[w]_i + P[w]_{i-1})f(w_{i+1} {-} w_i)
	\ge \frac{\beta}{4}\sumim |w_{i+1} {-} w_i|\log\left(1 + \frac{|w_{i+1} {-} w_i|}{\tau}\right) - C
\end{equation}
with a constant $C>0$ independent of $x$ and $\tau$. 

We further have by \eqref{dvp1} and by the monotonicity of $f$ and $g$ that
\begin{align}\nonumber
	&\frac1\tau\big(v_{i{+}1} - 2v_i  +  v_{i{-}1}\big)f(w_{i+1} {-} w_i)
	= \Big((u_{i{+}1} - u_i) - (v_{i{+}1} - v_i)\Big)f(w_{i+1} {-} w_i)\\ \label{difv}
	&\qquad \ge \tau (v_{i+1} - u_{i+1})f(w_{i+1} {-} w_i) \ge -C\tau,
\end{align}
where in the last step we employed estimates \eqref{Lam}-\eqref{vLam}, and the constant $C>0$ depends on $U$ 
from~\eqref{Lam}.

The boundary source term
$$
\ipo b^*(x)(u^*_{i+1} - u^*_i)f(w_{i+1} - w_i) \dd s(x)
$$
is bounded by a constant by virtue of Hypothesis~\ref{hy2}, while the boundary term on the left-hand side
$$
\ipo b^*(x)(u_{i+1} - u_i)f(w_{i+1} - w_i) \dd s(x)
$$
is non-negative by monotonicity of both functions $f$ and $g$. Using \eqref{dp4}-\eqref{difv} we thus get from \eqref{dp2}
\begin{align}\nonumber
	&\sumim \io |w_{i+1} - w_i|\log\left(1 + \frac{|w_{i+1} {-} w_i|}{\tau}\right)\dd x\\ \label{dp5}
	&\qquad + \sumim\io \Big(\kappa(x,\theta_{i{+}1})\nabla u_{i{+}1}{-}\kappa(x,\theta_i) \nabla u_i {+}(\kappa(x,\theta_{i{+}1}){-} \kappa(x,\theta_{i}))\bfnu\Big)\cdot \nabla f(w_{i+1} - w_i)\dd x \le C
\end{align}
with a constant $C>0$ independent of $\tau$. We further have
\begin{align*}
	&\Big(\kappa(x,\theta_{i+1})\nabla u_{i+1}-\kappa(x,\theta_i) \nabla u_i\Big)\cdot \nabla f(w_{i+1} {-} w_i)\\
	&= f'(w_{i+1} {-} w_i)\Bigg(\Big(\big(\kappa(x,\theta_{i+1}) {-} \kappa(x,\theta_{i})\big)\nabla u_i + \kappa(x,\theta_{i+1})\nabla (u_{i+1} {-} u_i)\Big)\\
	&\qquad \times \Big(\big(g'(u_{i+1}) {-} g'(u_{i})\big)\nabla u_i + g'(u_{i+1})\nabla (u_{i+1} {-} u_i)\Big)\Bigg)\\
	&= f'(w_{i+1} {-} w_i)\Bigg(g'(u_{i+1})\kappa(x,\theta_{i+1})|\nabla (u_{i+1} {-} u_i)|^2 + \big(g'(u_{i+1}) {-} g'(u_{i})\big)\big(\kappa(x,\theta_{i+1}) {-} \kappa(x,\theta_{i})\big)|\nabla u_i|^2\\
	&\qquad + \Big(\kappa(x,\theta_{i+1})\big(g'(u_{i+1}) {-} g'(u_{i})\big) + g'(u_{i+1})\big(\kappa(x,\theta_{i+1}) {-} \kappa(x,\theta_{i})\big)\Big)\nabla u_i\cdot \nabla (u_{i+1} {-} u_i)\Bigg).
\end{align*}
The functions $\kappa$ and $g'$ are bounded and Lipschitz continuous, and
\begin{equation}\label{dp9}
	f'(w_{i+1} - w_i) = \frac{\tau}{(\tau + |w_{i+1} - w_i|)^2}.
\end{equation}
Moreover, since $\theta_i = P[w]_i$ admits a representation similar to \eqref{de3}, by Hypothesis~\ref{hy2}, estimate \eqref{GiL}, and the Lipschitz continuity of the time-discrete play implied by \eqref{de4a}, we obtain
\begin{equation}\label{dp10}
	|\kappa(x,\theta_{i+1}) {-} \kappa(x,\theta_{i})| \le C|\theta_{i+1} {-} \theta_i| \le C|w_{i+1} {-} w_i|,
\end{equation}
whereas, from assumption \eqref{hg},
\begin{equation}\label{dp11}
	|g'(u_{i+1}) {-} g'(u_{i})| \le \bar{g}(U)|u_{i+1} {-} u_i| \le \bar{g}(U)(g_*(U))^{-1}|w_{i+1} {-} w_i|.
\end{equation}
Thus, employing the Cauchy-Schwarz and Young's inequalities and assumption \eqref{hg} on $g$, we conclude that there exist constants $\delta>0$ and $K>0$ independent of $\tau$ such that
\begin{equation}\label{dp8}
	\begin{aligned}
		\Big(\kappa(x,\theta_{i+1})\nabla u_{i+1}&-\kappa(x,\theta_i) \nabla u_i\Big)\cdot \nabla f(w_{i+1} {-} w_i) \\
		&\ge \delta f'(w_{i+1} {-} w_i)|\nabla (u_{i+1} {-} u_i)|^2 - K\frac{\tau}{(\tau + |u_{i+1} {-} u_i|)^2}|u_{i+1} {-} u_i|^2|\nabla u_i|^2 \\
		&\ge \delta f'(w_{i+1} {-} w_i)|\nabla (u_{i+1} {-} u_i)|^2 -\tau K |\nabla u_i|^2.
	\end{aligned}
\end{equation}
The remaining term $(\kappa(x,\theta_{i{+}1}){-} \kappa(x,\theta_{i}))\bfnu\cdot \nabla f(w_{i+1} - w_i)$ in \eqref{dp5} can be treated as follows. By \eqref{hg} we have
$$
\begin{aligned}
	|\nabla f(w_{i+1} - w_i)| &= f'(w_{i+1} - w_i)\big|g'(u_{i+1})\nabla u_{i+1} - g'(u_i)\nabla u_i\big| \\
	&= f'(w_{i+1} - w_i)\big|(g'(u_{i+1})-g'(u_i))\nabla u_i + g'(u_{i+1})(\nabla u_{i+1}-\nabla u_i)\big| \\
	&\le C f'(w_{i+1} - w_i)\Big(|u_{i+1}-u_{i}||\nabla u_i| + |\nabla(u_{i+1} - u_{i})|\Big),
\end{aligned}
$$
and employing \eqref{dp10} and arguing as in \eqref{dp8} we get
\begin{equation}\label{dp7}
	|\kappa(x,\theta_{i{+}1}){-} \kappa(x,\theta_{i})| |\nabla f(w_{i+1} - w_i)|
	\le \frac\delta{2}f'(w_{i+1} {-} w_i)|\nabla (u_{i+1} {-} u_i)|^2 + C \tau\left(1{+} |\nabla u_i|^2\right)
\end{equation}
with some constant $C>0$ independent of $\tau$. As a consequence of \eqref{hg}, \eqref{dp5}--\eqref{dp7}, and \eqref{energy}, we thus have the crucial estimate
\begin{equation}\label{dp6}
	\sumim \io |u_{i+1} - u_i|\log\left(1 + \frac{|u_{i+1} {-} u_i|}{\tau}\right)\dd x \le C\left(1+ \tau \sumiz \io|\nabla u_i|^2 \dd x\right) \le C
\end{equation}
with a constant $C>0$ independent of $\tau$.


\section{Proof of Theorem~\ref{t1}}\label{proof}

The estimates we have derived in Sections~\ref{unif} and \ref{conv} are sufficient for passing to the limit as $\tau\to 0$. For the time-discrete sequence $\{u_i\}$, we repeat literally the compactness argument in Sobolev and Orlicz spaces for piecewise linear and piecewise constant interpolates from \cite[Sections~5~and~6]{perme}. The most delicate step is the convergence of the hysteresis terms. Estimates \eqref{energy}, \eqref{Lam}, and \eqref{dp6} guarantee that the piecewise linear interpolates are bounded in the space
$$
\mathcal{B} \coloneqq L^\infty(\Omega\times (0,T))\cap \XX \cap L^1(\Omega; W^{1,\Phi_{log}}(0,T)),
$$
where
$$
\XX = \{u \in L^2(\Omega\times (0,T)) : \nabla u \in L^2(\Omega\times (0,T);\real^N)\}$$
and
$$
W^{1,\Phi_{log}}(0,T) = \{u \in L^1(0,T): \dot u \in L^{\Phi_{log}}(0,T)\}.
$$
From \cite[Proposition~6.1]{perme}, the space $\mathcal{B}$ is compactly embedded in $L^1(\Omega;C[0,T])$, and this is enough to ensure the convergence of the hysteresis terms in view of Proposition~\ref{pc1}.
As for the piecewise linear and piecewise constant interpolates of the sequence $\{v_i\}$, their convergence follows from the (strong) convergence of the interpolates of $\{u_i\}$ (compare in \eqref{dv1}) combined with the Lebesgue dominated convergence theorem. The limits are the desired solution to Problem~\eqref{pde}--\eqref{ie3}, and this concludes the proof of Theorem~\ref{t1}.


\section{Conclusions}

In this paper, we have investigated a model for unsaturated flow in a viscoelastic porous medium exhibiting degenerate hysteresis in the pressure-saturation relationship, coupled with gravity-driven moisture flux. This extends previous work on degenerate diffusion in porous media with hysteresis by incorporating the effects of solid matrix deformation and gravity, which introduce significant challenges in the analysis. The key difficulty lies in obtaining a priori bounds for the solutions, which we address by employing a hysteresis variant of the Moser iteration technique. Specifically, we require a condition on the Preisach density, namely, that it decays sufficiently slowly at infinity (Hypothesis~\ref{hrho}). This condition allows us to establish a uniform upper bound for the solutions. Building upon this estimate, we leverage the convexity and compactness arguments developed in \cite{perme} to prove the existence of a weak solution to the coupled system. Our approach involves a time discretization scheme, the derivation of estimates independent of the time step, and a passage to the limit.

The results presented here contribute to a deeper understanding of flow phenomena in deformable porous media, relevant to various applications in hydrogeology and related fields. Future work could explore the impact of more complex constitutive models for the solid matrix, as well as the development of efficient numerical methods for the computation of solutions.


\end{document}